\documentclass{article}
\usepackage{graphicx}
\usepackage{epstopdf}
\usepackage{amsmath}
\usepackage{amssymb}
\usepackage{bm}
\usepackage{bbm}
\usepackage{enumerate}
\usepackage[T1]{fontenc}
\usepackage[latin1]{inputenc}
\usepackage{hyperref}
\usepackage{color}
\usepackage{curves}
\usepackage{tikz}
\usetikzlibrary{calc}
\usepackage{wrapfig}
\usepackage{ntheorem}
\usepackage{array}
\usepackage{arydshln}
\usepackage{multirow}
\usepackage{float}

\title{Integer matrices that are not copositive have
certificates of less than quadratic complexity}
\author{Timo Vilkas\thanks{Lunds Universitet}
               }
\theoremstyle{break}
\newtheorem{theorem}{Theorem}[section]

\newtheorem{lemma}{Lemma}

\theorembodyfont{\upshape}

\newtheorem*{remark}{Remark}

\newenvironment{proof}{{\sc Proof:}}{\vspace{-1em}~\hfill $\square$\vspace{2em}}
\newenvironment{nproof}[1]{{\sc Proof #1:}}{\vspace{-1em}~\hfill $\square$\vspace{2em}}

\newcommand\N{\mathbb{N}}
\newcommand\R{\mathbb{R}}
\newcommand\Z{\mathbb{Z}}

\renewcommand\epsilon{\varepsilon}
\newcommand{\ve}[1]{\mathbf{#1}}
\newcommand{\sgn}{\mathop{\mathrm{sgn}}}
\definecolor{darkblue}{rgb}{0,0,.5}
\hypersetup{colorlinks=true, breaklinks=true, linkcolor=blue, 
citecolor=darkblue, menucolor=blue, urlcolor=blue}
\begin{document}
\newpage
\maketitle
\begin{abstract}
    A symmetric matrix $M$ in $\R^{n\times n}$ is called copositive if the
    corresponding quadratic form $Q(\ve{x})=\ve{x}^TM\ve{x}$ is
    non-negative on the closed first orthant $\R^n_{\geq 0}$. If
    the matrix fails to be copositive there exists some certificate
    $\ve{x}\in\R^n_{\geq 0}$ for which the quadratic form is negative.
    Due to the scaling property $Q(\lambda\,\ve{x})=\lambda^2\,Q(\ve{x})$ for
    $\lambda\in\R$, we can find such certificates in every neighborhood
    of the origin but their properties depend on $M$ of course and are
    difficult to describe. If $M$ is an integer matrix, however, we are
    guaranteed certificates of a complexity that is at most a constant times
    the binary encoding length of the matrix raised to the power $\tfrac32$.
\end{abstract}



\section{Introduction}

Let $M=(m_{ij})_{i,j}\in\R^{n\times n}$ be a symmetric real-valued matrix. As is known $M$ is called
positive semi-definite if all its eigenvalues are non-negative or equivalently the
corresponding quadratic form is non-negative, i.e.\ $Q(\ve{x}):=\ve{x}^TM\ve{x}\geq 0$ for all vectors
$\ve{x}\in\R^n$.

$M$ is in turn called {\itshape copositive} if this condition holds true for all vectors with 
non-negative entries, i.e.
 \begin{equation}\label{def}
   \ve{x}^TM\ve{x}\geq 0 \text{ for all vectors } \ve{x}\in\R_{\geq0}^n.
 \end{equation}
Obviously, positive semi-definite matrices are copositive, as are all symmetric non-negative
matrices, since $m_{ij}\geq 0$ for all $1\leq i,j\leq n$ implies (\ref{def}).

However, there are symmetric non-negative matrices, which are not positive semi-definite. 
Hence copositive matrices are a proper subset of all symmetric matrices (having negative diagonal
entries trivially renders copositivity impossible) and a proper superset of the positive 
semi-definite matrices in $\R^{n\times n}$ for $n\geq 2$. For $n=1$ positive semi-definiteness
and copositivity correspond to non-negativity and are thus equivalent.

$$M:=\left(
\begin{array}{cc} 0 & 1 \\ 1 & 0 \end{array}
\right)$$
has eigenvalues $\{-1,1\}$ but is non-negative showing that in $\R^{2\times 2}$ copositivity
is strictly weaker than positive semi-definiteness. Putting $M$ as upper left corner in an $n\times n$
zero matrix will establish the same for higher dimensions.

Murty and Kabadi showed that checking whether a given square matrix is not copositive is NP-complete. They actually showed this hardness result for integer matrices, cf.\ \cite{1}. This immediately implies that the decision about copositivity of an integer square matrix is coNP-complete. In fact, their results can be used to guarantee vectors that witness the non-copositivity of an integer square matrix with polynomial complexity.

\section{Finding relatively simple certificates}

The main result to be established here is to give a concrete upper bound for the minimal complexity of certificates
for integer matrices which are not copositive in form of the following statement:

\begin{theorem}\label{thm}
Let $M\in\Z^{n\times n}$ be a symmetric integer-valued matrix. If there exists some $\ve{x}\in\R_{\geq0}^n$
such that $\ve{x}^TM\ve{x}< 0$, i.e.\ a certificate for $M$ not being copositive, a vector $\ve{y}\in\R_{\geq0}^n$
can be found such that $\ve{y}^TM\ve{y}< 0$ and the binary encoding length of $\ve{y}$ is at most $17$ times
that of $M$ to the power $\tfrac32$.
\end{theorem}

On the way to achieve this result a couple of lemmas are needed, some of which were already 
sketched by Murty and Kabadi.

\begin{lemma}
Given $M\in\R^{n\times n}$, a symmetric real-valued matrix, let us define the
following minimization problem:
\begin{flalign}\label{min}\hspace*{10mm}\begin{split}\begin{array}{l l} \text{minimize}   & Q(\ve{x})=\ve{x}^TM\ve{x}\\
                  \text{subject to} & \ve{x}\in[0,1]^n.
\end{array}\end{split}&\end{flalign}
For an optimal solution $\ve{\bar{x}}$ to {\upshape(\ref{min})}, there exist vectors
$\ve{\bar{y}},\ve{\bar{u}},\ve{\bar{v}}\in\R_{\geq0}^n$ such that
\begin{equation}\label{cl}\left(\begin{array}{c}\ve{\bar{u}}\\\ve{\bar{v}}\end{array}\right) -
 \left(\begin{array}{cc} M & I\\-I& 0\end{array}\right)\cdot
 \left(\begin{array}{c}\ve{\bar{x}}\\\ve{\bar{y}}\end{array}\right)=
 \left(\begin{array}{c}\ve{0}\\\ve{e}\end{array}\right)\quad\text{and }
 \end{equation} 
\begin{equation}\label{cl2}
 (\,\ve{\bar{u}}^T\,,\ \ve{\bar{v}}^T\,)\cdot 
 \left(\begin{array}{c}\ve{\bar{x}}\\\ve{\bar{y}}\end{array}\right)=0,
 \end{equation}
 where $I$ denotes the $n\times n$ identity matrix and $\ve{e}\in\R^n$ the vector of all ones.
\end{lemma}
\begin{proof} First of all, $[0,1]^n$ is bounded and closed, hence compact, and the quadratic form
$Q$ a continuous function on $\R^n$. Therefore it attains its minimum $\gamma:=\min_{\ve{x}\in[0,1]^n}Q(\ve{x})\in\R$
and (\ref{min}) has an optimal solution. 

\noindent Let $\ve{\bar{x}}\in[0,1]^n$ be such that $Q(\ve{\bar{x}})=\gamma$.
From quadratic programming, it is known that an optimal solution $\ve{\bar{x}}$ to the quadratic program
\begin{flalign}\label{QP}\hspace*{10mm}\begin{split}\begin{array}{l l} \text{minimize}   & 
                        Q(\ve{x})=\ve{c}^T\ve{x}+\tfrac12\ve{x}^TD\ve{x}\\
                  \text{subject to} & A\ve{x}\geq \ve{b}\\
                  \text{and} & \ve{x}\in \R_{\geq0}^n,
\end{array}\end{split}&\end{flalign}
where $\ve{b},\ve{c}\in \R^n$ and $A,D\in\R^{n\times n}$, is also an optimal solution to the linear program
\begin{flalign}\label{LP}\hspace*{10mm}\begin{split}\begin{array}{l l} \text{minimize}   & 
                        (\ve{c}^T+\ve{\bar{x}}^TD)\ve{x}\\
                  \text{subject to} & A\ve{x}\geq \ve{b}\\
                  \text{and} & \ve{x}\in \R_{\geq0}^n,
\end{array}\end{split}&\end{flalign}
see for example Thm.\ 1.12 in \cite{2}. It is easy to check that (\ref{min}) is equivalent to
(\ref{QP}) if we choose $\ve{c}=\ve{0},\ D=M,\ A=-I$ and $\ve{b}=-\ve{e}$. Proceeding to the linear program, it is thus
equivalent to
\begin{flalign}\label{LP2}\hspace*{10mm}\begin{split}\begin{array}{l l} \text{maximize}   & 
                        -\ve{\bar{x}}^TM\ve{x}\\
                  \text{subject to} & -\ve{x}\geq -\ve{e}\\
                  \text{and} & \ve{x}\in \R_{\geq0}^n,
\end{array}\end{split}&\end{flalign}
consequently having the value $-\gamma$. Rewritten as cone program, this reads
\begin{flalign}\label{CP}\hspace*{10mm}\begin{split}\begin{array}{l l} \text{maximize}   & 
                        \langle-M\ve{\bar{x}},\ \ve{x}\rangle\\
                  \text{subject to} & \ve{e}-I\ve{x}\in \R_{\geq0}^n\\
                  \text{and} & \ve{x}\in \R_{\geq0}^n.
\end{array}\end{split}&\end{flalign}
Noting that we have interior points, e.g.\ $\ve{x}=\tfrac12\,\ve{e}$, and that the cone $\R_{\geq0}^n$
is self-dual allows for another transformation. Duality theory tells us that the dual problem
\begin{flalign}\label{DP}\hspace*{10mm}\begin{split}\begin{array}{l l} \text{minimize}   & 
                        \langle\ve{e},\ \ve{y}\rangle\\
                  \text{subject to} & I\ve{y}+M\ve{\bar{x}}\in \R_{\geq0}^n\\
                  \text{and} & \ve{y}\in \R_{\geq0}^n,
\end{array}\end{split}&\end{flalign}
is also feasible and has the same value $-\gamma$, see for example Thm.\ 4.7.1 in \cite{3}. If we
denote an optimal solution to the dual problem by $\ve{\bar{y}}$ and let $\ve{\bar{u}}:=\ve{\bar{y}}+M\ve{\bar{x}}$, 
$\ve{\bar{v}}:=\ve{e}-\ve{\bar{x}}$, we have indeed $\ve{\bar{u}},\ve{\bar{v}},\ve{\bar{x}},\ve{\bar{y}}\in\R_{\geq0}^n$,

\begin{equation*}\left(\begin{array}{c}\ve{\bar{u}}\\\ve{\bar{v}}\end{array}\right) -
 \left(\begin{array}{cc} M & I\\-I& 0\end{array}\right)\cdot
 \left(\begin{array}{c}\ve{\bar{x}}\\\ve{\bar{y}}\end{array}\right)=
 \left(\begin{array}{c}\ve{0}\\\ve{e}\end{array}\right)\quad\text{and}
\end{equation*}
\begin{equation*}
 (\,\ve{\bar{u}}^T\,,\ \ve{\bar{v}}^T\,)\cdot 
 \left(\begin{array}{c}\ve{\bar{x}}\\\ve{\bar{y}}\end{array}\right)
 =\ve{\bar{y}}^T\ve{\bar{x}}+\ve{\bar{x}}^TM\ve{\bar{x}}+\ve{e}^T\ve{\bar{y}}-\ve{\bar{x}}^T\ve{\bar{y}}
 =\gamma-\gamma=0,
\end{equation*}
which establishes the claim.
\end{proof}

\vspace{0.5em}\noindent
For a system of linear equations in non-negative variables such as
\begin{equation}\label{BFS}\begin{split}
A\ve{s}=\ve{b}\ \\
\ve{s}\in\R^l_{\geq0}\end{split}
\end{equation}
where $A=(A_1,\dots,A_l)\in\R^{k\times l},\ \ve{b}\in\R^k$, a vector $\ve{s}\in\R^l$ is called a solution
if $A\ve{s}=\ve{b}$, feasible if $\ve{s}\in\R^l_{\geq0}$ and a {\itshape basic feasible solution} (abbreviated:
BFS) if it satisfies (\ref{BFS}) and the set of columns $\{A_j,\; s_j>0\}$ is linearly independent. Thm.\ 3.1
in \cite{4} states that the basic feasible solutions are precisely the extreme points of the convex
set of feasible solutions.

\begin{lemma}
Let $M\in\R^{n\times n}$ again be a symmetric real-valued matrix and consider the system
of linear equations in non-negative variables
\begin{equation}\label{le}
A\ve{s}=\ve{b},\text{ where } A:={\setlength{\extrarowheight}{3.5pt}   
\left(\begin{array}{c:c:c}  -M & -I& \multirow{2}{*}{\ I\,}
\\ \cdashline{1-2} I& 0 
\end{array}\right)}\in\R^{2n\times4n},\ \ve{b}:=\left(\begin{array}{c}\ve{0}\\\ve{e}\end{array}\right)\in\R^{2n}.
\end{equation}
Then there exist $\ve{\bar{u}},\ve{\bar{v}},\ve{\bar{x}},\ve{\bar{y}}\in\R_{\geq0}^n$ such that
$$\ve{s}:=\left(\begin{array}{c}\ve{\bar{x}}\\\ve{\bar{y}}\\\ve{\bar{u}}\\\ve{\bar{v}}\end{array}\right)$$
is a BFS to (\ref{le}) and (\ref{cl2}) is also satisfied.
\end{lemma}
\begin{proof}
First note that being a solution of (\ref{le}) is equivalent to (\ref{cl}). Furthermore, from (\ref{cl}) and
(\ref{cl2}) it follows (by multiplying (\ref{cl}) with $(\ve{\bar{x}}^T,\ve{\bar{y}}^T)$ from the left):
\begin{equation}\label{opt}-\ve{\bar{x}}^TM\ve{\bar{x}}=\ve{\bar{y}}^T\ve{e}.\end{equation}
With $\ve{\bar{x}}^{(0)}$ being an optimal solution to problem (\ref{min}) and the corresponding vectors
$\ve{\bar{u}}^{(0)},\ve{\bar{v}}^{(0)},\ve{\bar{y}}^{(0)}\in\R_{\geq0}^n$ defined as in the foregoing lemma, we know that
$$\ve{s}^{(0)}:=\left(\begin{array}{c}\ve{\bar{x}}^{(0)}\\\ve{\bar{y}}^{(0)}\\
                                      \ve{\bar{u}}^{(0)}\\\ve{\bar{v}}^{(0)}\end{array}\right)$$
is a solution to (\ref{le}) which also satisfies (\ref{cl2}). However, it is not guaranteed that this
is an extreme point in the set of feasible solutions.

If not, we proceed as follows:\\
Assume $\ve{s}^{(0)}$ is no extreme point, then there exist distinct feasible solutions $\ve{t}^{(1)},\ve{t}^{(2)}$ and
$\alpha\in(0,1)$ s.t.\ $\ve{s}^{(0)}=\alpha\,\ve{t}^{(1)}+(1-\alpha)\,\ve{t}^{(2)}$.
From (\ref{cl2}) and the non-negativity,
$\ve{\bar{u}}^{(0)},\ve{\bar{v}}^{(0)},\ve{\bar{x}}^{(0)},\ve{\bar{y}}^{(0)}\in\R_{\geq0}^n$,
it follows that only one coordinate in each of the pairs
$$\{(\bar{x}^{(0)}_i,\bar{u}^{(0)}_i),(\bar{y}^{(0)}_i,\bar{v}^{(0)}_i);\;1\leq i\leq n\}$$ can be strictly
positive. If one writes $$\ve{t}^{(i)}=:\left(\begin{array}{c}\ve{x}^{(i)}\\\ve{y}^{(i)}\\
\ve{u}^{(i)}\\\ve{v}^{(i)}\end{array}\right)\text{ for }i=1,2,$$
the non-negativity of $\ve{t}^{(1)},\ve{t}^{(2)}$ and $\alpha\in(0,1)$ together imply that
$s^{(0)}_j=0$ forces $t^{(1)}_j=t^{(2)}_j=0$. Hence the orthogonality relation from (\ref{cl2}) also holds for
both $\ve{t}^{(1)}$ and $\ve{t}^{(2)}$. Consequently, we get 
\begin{eqnarray*}(\ve{\bar{x}}^{(0)})^TM\ve{\bar{x}}^{(0)}&=&-\ve{e}^T\ve{\bar{y}}^{(0)}
=\alpha\cdot\big(-\ve{e}^T\ve{y}^{(1)}\big)+(1-\alpha)\cdot\big(-\ve{e}^T\ve{y}^{(2)}\big)\\
&=&\alpha\cdot\big((\ve{x}^{(1)})^TM\ve{x}^{(1)}\big)+(1-\alpha)\cdot\big((\ve{x}^{(2)})^TM\ve{x}^{(2)}\big)\\
&\geq&(\ve{\bar{x}}^{(0)})^TM\ve{\bar{x}}^{(0)}
\end{eqnarray*}
by optimality of $\ve{\bar{x}}^{(0)}$, where the first and third equality follow from the consideration in 
(\ref{opt}). This implies
$(\ve{x}^{(i)})^TM\ve{x}^{(i)}=(\ve{\bar{x}}^{(0)})^TM\ve{\bar{x}}^{(0)}$, i.e.\ $\ve{t}^{(1)},\ve{t}^{(2)}$ also
feature optimal solutions to (\ref{min}) in their first $n$ coordinates.

Define the line $l:\ \ve{s}(r)=\ve{s}^{(0)}+r\cdot(\ve{t}^{(2)}-\ve{t}^{(1)}),\ r\in\R$. Linearity guarantees
that every $$\ve{s}=\left(\begin{array}{c}\ve{x}\\\ve{y}\\\ve{u}\\\ve{v}\end{array}\right)
\in l\cap\R_{\geq0}^{4n}$$ is a feasible solution to (\ref{le}) which also satisfies (\ref{cl2}), because
$s^{(0)}_j=0$ implies $s_j=0$. As above, from this we can deduce
$$\ve{x}^TM\ve{x}=-\ve{e}^T\ve{y}=-\ve{e}^T\ve{\bar{y}}^{(0)}+r\cdot
\underbrace{(-\ve{e}^T\ve{y}^{(2)}+\ve{e}^T\ve{y}^{(1)})}_{=\,0}=(\ve{\bar{x}}^{(0)})^TM\ve{\bar{x}}^{(0)},$$
i.e.\ $\ve{x}$ is another optimal solution to (\ref{min}).

Since $\ve{t}^{(1)}\neq\ve{t}^{(2)},\ r$ can be chosen in such a way that $$\ve{s}=\ve{s}(r)\in\R_{\geq0}^{4n}
\text{ and there exists some index $j$ with }s_j=0\neq s^{(0)}_j.$$ Use this feasible solution to define 
$\ve{s}^{(1)}:=\ve{s}$. As $\ve{s}^{(0)}$ has not more than $2n$ non-zero
coordinates, this procedure (when iterated) must stop, yielding a point $\ve{s}^{(k)}\in\R_{\geq0}^{4n}$
that is an extreme point in the set of feasible solutions to (\ref{le}). It will also satisfy (\ref{cl2}), which
in turn implies that $\bar{\ve{x}}^{(k)}$ is again optimal for (\ref{min}).
\end{proof}

\begin{lemma}
Let $M\in\Z^{n\times n}$ now be a symmetric integer-valued matrix.
The optimal value in {\upshape(\ref{min})} is either 0 (iff $M$ is copositive) or at most $-2^{-2L+1}$, where
$L$ denotes the binary encoding length of $M$.
\end{lemma}
\begin{proof}
The statement about a copositive matrix $M$ follows directly from the definition and $\ve{0}^TM\ve{0}=0$. 

In the other case, let 
$$\ve{s}:=\left(\begin{array}{c}\ve{\bar{x}}\\\ve{\bar{y}}\\\ve{\bar{u}}\\\ve{\bar{v}}\end{array}\right)$$
be as guaranteed by the above lemma, which means among other things that $\ve{\bar{x}}$ is optimal for 
(\ref{min}). Since $\ve{s}=(s_j)_{j=1}^{4n}$ is a BFS to (\ref{le}), the set $\mathcal{S}:=\{A_j,\; s_j>0\}$ 
of not more than $2n$ columns is linearly independent. The matrix $A$ clearly has rank $2n$, we can thus 
choose additional columns to get a superset of $\mathcal{S}$ which forms a base of $\R^{2n}$. Deleting the
unchosen columns in $A$ and corresponding zero entries in $\ve{s}$ gives an invertible $2n\times 2n$ submatrix
$B$ of $A$ and a subvector $\ve{\tilde{s}}\in\R^{2n}_{\geq0}$ of $\ve{s}$ such that $B\ve{\tilde{s}}=\ve{b}$.
Cramer's rule applies and gives 
\begin{equation}\label{cram}\tilde{s}_i=\frac{\det(B_i,b)}{\det B},\end{equation}
where $(B_i,b)$ denotes the matrix in which the $i$th column of $B$ has been replaced by $\ve{b}$. The fact
that all entries in $A$, hence $B$, and $\ve{b}$ are integers, implies that the determinants in (\ref{cram}) are
integers too and $\det B\neq0$ as $B$ is invertible. Consequently, the non-zero entries of $\ve{s}$ are at
least $|\det B|^{-1}$ due to non-negativity and $\det(B_i,b)\in\Z$.

To finish off the proof of this lemma, whose final part is essentially a concatenation of ideas from section
15.2 in \cite{4}, it is left to show that $|\det B|\leq 2^{2L-1}$, since as in the foregoing lemma, we have
$\ve{\bar{x}}^TM\ve{\bar{x}}=-\ve{e}^T\ve{\bar{y}}$. So in the case of $M$ not being copositive, the left
hand side is negative forcing positive entries in $\ve{\bar{y}}$, hence $\ve{s}$ hence $\ve{\tilde{s}}$. Those
are in turn at least $|\det B|^{-1}$, which by non-negativity of $\ve{\bar{y}}$ implies 
$\ve{\bar{x}}^TM\ve{\bar{x}}\leq-|\det B|^{-1}$.\\[1em]
First of all, the binary encoding length of the original matrix has to be determined. Clearly, the symmetry
allows to encode the upper triangular part only. To store $M$ in the upper triangular part of an $n\times n$ array we need 
\begin{equation}\label{com}
L:=\sum_{1\leq i\leq j\leq n}\Big(\big\lceil \log_2(|m_{ij}|+1)\big\rceil+1\Big)
\end{equation}
bits, since $\lceil \log_2(|m_{ij}|+1)\rceil$ bits are needed to represent $|m_{ij}|$ if larger than $0$ and 
one bit for its sign, just one bit if $m_{ij}=0$.

Having the form of $A$ in mind (see (\ref{le})), we can expand the determinant of the submatrix $B$ 
with respect to first columns to the right then rows below $M$ in $A$
such that $\det(B)=\pm \det(B')$, where $B'$ is a $l\times l$ submatrix of $M$. If $S_l$ denotes the group
of permutations on $\{1,\dots,l\}$ and $B'=(b_{ij})_{i,j}$, one gets using Leibniz' formula:
\begin{eqnarray*}
|\det(B')|&=&\left|\sum_{\sigma\in S_l}(-1)^{\sgn(\sigma)}\ b_{1,\sigma(1)}\cdot\ldots\cdot b_{l,\sigma(l)}\right|\\
          &\leq&\sum_{\sigma\in S_l}|b_{1,\sigma(1)}|\cdot\ldots\cdot |b_{l,\sigma(l)}|
          \quad\leq\quad\prod_{i=1}^l (|b_{i1}|+\ldots+|b_{il}|)\\
          &\leq&\prod_{i=1}^n (|m_{i1}|+\ldots+|m_{in}|)
          \quad\leq\quad\prod_{1\leq i,j\leq n} (|m_{ij}|+1)\\
          &\leq&2^{\sum_{1\leq i,j\leq n} \log_2(|m_{ij}|+1)}\quad\leq\quad 2^{2L-1},
\end{eqnarray*}
where the last inequality follows directly from the consideration in (\ref{com}).
\end{proof}

Having prepared all those auxiliary results, we can finally proceed to proving the central conclusion.\\[1em]
\begin{nproof}{(of Thm.\ \ref{thm})}
To begin with, it is obvious that the complexity of $M$ is at least the number of entries necessary to represent
it in an array, i.e.\ 
$$L\geq\#\{(i,j),\;1\leq i\leq j\leq n\}=\tfrac {n\,(n+1)}{2},\quad\text{thus}\quad n\leq \sqrt{2L}.$$
Let $\bar{\ve{x}}$ be an optimal solution to (\ref{min}). By the lemma above, we know that the corresponding
value of the quadratic form is
$Q(\bar{\ve{x}})=\bar{\ve{x}}^TM\bar{\ve{x}}\leq-2^{-2L+1}$. Denote by $d:=\max_{i,j}|m_{ij}|$ the largest
entry of $M$ in terms of absolute value and note that $d\in\N$ since $M$ cannot be the zero matrix. 
Next, let us define $\ve{x}^{\ast}:=2^{2L-1}\cdot\bar{\ve{x}}$ and finally the vector $\ve{y}\in\R^n_{\geq0}$ by
\begin{equation}\label{y}
y_j:=\frac{1}{4dn^2}\big\lceil4dn^2\cdot x_j^\ast\big\rceil,\ \text{for }1\leq j\leq n.
\end{equation}
Let $\lVert\,.\,\rVert$ denote the Euclidean norm on $\R^n$. Due to $\bar{\ve{x}}\in[0,1]^n$ we get $\lVert\bar{\ve{x}}\rVert
\leq\sqrt{n}$, $\lVert\ve{x}^{\ast}\rVert\leq2^{2L-1}\,\sqrt{n}$ and clearly $L\geq\lceil\log_2(d+1)\rceil\geq\log_2d$.

Note that $\ve{y}$ is a non-negative rational vector and since every coordinate consists of an integer part in
$\{0,\dots,2^{2L-1}\}$ and a fractional part which is given by a numerator and a denominator in $\{0,\dots,4dn^2\}$,
its binary complexity is not larger than
\begin{align}\label{cv}\begin{split}
n\,\Big(\big\lceil\log_2(2^{2L-1}+1)\big\rceil+2\,\big\lceil\log_2(4dn^2+1)\big\rceil\Big)\\
&\hspace{-5cm}\leq n\,\Big(2L+2\,\big(\log_2(4dn^2)+1\big)\Big)\\
&\hspace{-5cm}\leq \sqrt{2L}\,\Big(2L+2\,\big(2+\log_2d+(\log_2L+1)+1\big)\Big)\\
&\hspace{-5cm}\leq \sqrt{2L}\,(4L+2\log_2L+8)\\
&\hspace{-5cm}\leq \sqrt{2L}\cdot12\,L\\
&\hspace{-5cm}\leq 17\,L^{3/2}.\end{split}
\end{align}
In the before last line the simple estimate $\log x+1\leq x$ for $x\geq 0$ and $L\geq 1$ was used.\\[1em]
Finally, it has to be checked that $Q(\ve{y})<0$. The definitions and estimates from above give:
$$\lVert\ve{y}\rVert\leq\big|\big|\ve{x}^\ast+\tfrac{1}{4dn^2}\,\ve{e}\big|\big|\leq 2^{2L-1}\sqrt{n}+\tfrac{1}{4dn^{3/2}}
\quad\text{and}\quad \lVert\ve{y}-\ve{x}^\ast\rVert\leq\tfrac{1}{4dn^{3/2}}.$$
Furthermore, the eigenvalues of $M$ are all of absolute value at most $dn$, since for every eigenvector
$\ve{v}=(v_i)_{i=1}^n$ corresponding to eigenvalue $\lambda$ it holds that:
$$|\lambda|=\frac{\max_i|(M\ve{v})_i|}{\max_i|v_i|}=\frac{\max_i\big|\sum_{j=1}^nm_{ij}\,v_j\big|}{\max_i|v_i|}
\leq \frac{\max_i\sum_{j=1}^nd\,|v_j|}{\max_i|v_i|} \leq dn.$$
Consequently, using these estimates we find:
\begin{eqnarray*}
\ve{y}^TM\ve{y}&=&\ve{y}^TM(\ve{y}-\ve{x}^\ast)+\ve{y}^TM\ve{x}^\ast\\
&=&\ve{y}^TM(\ve{y}-\ve{x}^\ast)+(\ve{y}-\ve{x}^\ast)^TM\ve{x}^\ast+(\ve{x}^\ast)^TM\ve{x}^\ast\\
&\leq&(2^{2L-1}\sqrt{n}+\tfrac{1}{4dn^{3/2}})\cdot\tfrac{dn}{4dn^{3/2}}+\tfrac{dn}{4dn^{3/2}}\cdot 2^{2L-1}\sqrt{n}
-2^{-2L+1}\cdot 2^{4L-2}\\
&\leq&(2^{2L-1}+\tfrac{1}{4d})\cdot\tfrac14+\tfrac14\cdot2^{2L-1}-2^{2L-1}\\
&\leq&2^{2L-1}\,(\tfrac12+\tfrac14-1)<0,
\end{eqnarray*}
where the last line follows from $d\geq1,\ L\geq1$.
\end{nproof}

\begin{remark}
\begin{enumerate}[(a)]
\item
  Choosing the discretization of $\ve{x}^\ast$ finer (i.e.\ with a spacing of $c\leq\tfrac{1}{4dn^2}$ in (\ref{y}))
  will make the above estimate only sharper, but at the same time increase the complexity. Choosing $l\in\N$
  minimal s.t.\ $2^l\geq 4dn^2$ and taking $2^{-l}$ as spacing, however, allows to write the fractional part
  of each coordinate as a sum of negative powers of $2$, i.e.\ $\{2^{-1},\dots,2^{-l}\}$ and thus
  reducing the summand in the estimate for the binary complexity coming from the pair numerator/denominator 
  from $2\,\lceil\log_2(4dn^2)+1\rceil$ to $l=\lceil\log_2(4dn^2)+1\rceil$. This leads to an
  overall complexity of not more than $10\,L^{3/2}$.
   
\item To evaluate the sharpness of this result, let us consider the following example. Let $k\in\N$,
  $$M:=\left(\begin{array}{cc} 2^{2k+2} &-2^{k+2}\\ -2^{k+2}&3\end{array}\right)$$
  and $Q(\ve{x})=\ve{x}^TM\ve{x}$ be again the corresponding quadratic form.
  This means for $\ve{x}=(1,0)^T$ one gets the value $Q(\ve{x})=2^{2k+2}>0$ and for $\ve{x}=(x,1)^T$ correspondingly 
  $Q(\ve{x})=2^{2k+2}x^2-2^{k+3}x+3=4\,(2^kx-1)^2-1$.\vspace*{0.5em}

  \begin{minipage}{0.50\textwidth}
  	The latter is smaller than $0$ if and only if $x\in(\tfrac{1}{2^{k+1}},\tfrac{3}{2^{k+1}})$.
  	Since $Q(\lambda\,\ve{x})=\lambda^2\,Q(\ve{x})$ for $\lambda\in\R$, this means that the 
  	certificates for $M$ not being copositive lie in the shaded area in the picture to the right.\\
  	This in turn implies that if we consider $\ve{y}=(p,q)^T\in\mathbb{Q}^2_{\geq 0}$, a certificate
  	with rational entries, either the denominator appearing in $q$ is at least $2^k$ or the product
  	of the integer part of $p$ and the denominator in $q$ is. Either way, the binary complexity
  	of $\ve{y}$ is at least $k+1$.
  	Another look at $M$ reveals that the binary encoding length of this matrix is according to
  	(\ref{com}) precisely\\[0.2cm]\hspace*{0.4cm} $L=(2k+4)+(k+3)+3=3k+10.$
  \end{minipage}
  \hspace{1cm}
  \begin{minipage}{0.40\textwidth}
    \begin{figure}[H]
	\unitlength=0.8mm
	\begin{tikzpicture}(50.00,70.00)
		\fill[fill=blue!25] (0,0) -- (0.85,5.1) -- (2.55,5.1) -- cycle;  
		{\color[rgb]{.4,.4,.4}
			\put(-4.50,68.00){$y$}
			\put(49.00,-5.50){$x$}
			\put(-2.00,0.00){\vector(1,0){51.00}}
			\put(-1.50,60.00){\line(1,0){3.00}}
			\put(0.00,-2.00){\vector(0,1){70.00}}
			\put(20.00,-1.50){\line(0,1){3.00}}
			\put(-4.50,-5.50){$\small 0$}
			\put(17.50,-5.50){$\small 2^{-k}$}
			\put(-4.50,58.80){$\small 1$}}
		\multiput(4,60.00)(4,0){11}{\line(1,0){2}}
		{\color[rgb]{0,0,1} 
			\put(0.00,0.00){\line(1,6){11.00}}
			\put(0.00,0.00){\line(1,2){33.00}}}
	\end{tikzpicture}
	\end{figure}
  \end{minipage}

  Hence every certificate has a complexity which is at least linear in the encoding length of $M$.
\item Note that the extra factor $L^{1/2}$ in the estimate for the complexity (\ref{cv}) is coming from the size
  $n$ of the vector. If we fix the dimension, the result attained in Thm.\ \ref{thm} actually is that there exists a
  certificate with complexity at most $n\cdot 12\,L$, which is linear in $L$ and hence up to the constant factor
  tight according to the above example:\\ 
  Putting the matrix $M$ as the upper left corner of a zero matrix in
  $\R^{n\times n},$ for $n\geq 2$, will lead to an encoding length of $3k+10+\tfrac{n}{2}(n+1)-3$, since only the
  extra zeros in the upper triangular part have to be encoded. With $n$ fixed, the lower bound on the complexity
  of a certificate (which is $k+1+(n-2)$ by the same reasoning as above) is still linear in the complexity of $M$.
  
  If $n$ (which is known to be at most $\sqrt{2L}$) is not constant, i.e.\ not in $\mathcal O(L^0)$ but only in
  $\mathcal O(L^\delta)$ for some $\delta\in(0,\tfrac12],$ the established upper bound is superlinear, namely a constant times
  $L^{1+\delta}$, and it is not clear whether this is tight.
\end{enumerate}
\end{remark}
%
%


\vspace{0.5cm}

   {\sc \small \noindent Timo Vilkas\\
   Department of Statistics,\\
   Lund University School of Economics and Management\\
   Sweden.}\\
   timo.vilkas@stat.lu.se

\end{document}